\documentclass[12pt]{amsart}

\usepackage{amssymb}
\usepackage{hyperref}
\usepackage{color}
\allowdisplaybreaks
\usepackage{tikz}
\usepackage{enumitem}
\usepackage{ulem}

\newtheorem{thm}{Theorem}[section]
\newtheorem{lem}[thm]{Lemma}
\newtheorem{cor}[thm]{Corollary}
\newtheorem{prop}[thm]{Proposition}
\theoremstyle{definition}
\newtheorem{defn}[thm]{Definition}
\newtheorem{ex}[thm]{Example}
\newtheorem{exs}[thm]{Examples}
\newtheorem{rem}[thm]{Remark}

\numberwithin{equation}{section}

\newcommand{\thmref}[1]{Theorem~\textup{\ref{#1}}}

\newcommand{\lemref}[1]{Lemma~\textup{\ref{#1}}}
\newcommand{\propref}[1]{Proposition~\textup{\ref{#1}}}

\newcommand{\exref}[1]{Example~\textup{\ref{#1}}}

\newcommand{\midtext}[1]{\quad\text{#1}\quad}
\newcommand{\righttext}[1]{\quad\text{#1 }}
\renewcommand{\and}{\midtext{and}}
\renewcommand{\for}{\righttext{for}}
\newcommand{\all}{\righttext{for all}}

\newcommand{\N}{\mathbb N}
\newcommand{\Z}{\mathbb Z}
\newcommand{\C}{\mathbb C}
\newcommand{\T}{\mathbb T}

\newcommand{\KK}{\mathcal K}

\newcommand{\LL}{\mathcal L}
\newcommand{\OO}{\mathcal O}
\newcommand{\NO}{\mathcal{NO}}

\renewcommand{\a}{\alpha}
\renewcommand{\b}{\beta}
\renewcommand{\O}{\Omega}
\newcommand{\Th}{\Theta}
\renewcommand{\L}{\Lambda}

\DeclareMathOperator{\obj}{Obj}
\DeclareMathOperator{\mor}{Mor}
\DeclareMathOperator{\supp}{supp}
\DeclareMathOperator{\End}{End}

\DeclareMathOperator*{\spn}{span}

\newcommand{\<}{\langle}
\renewcommand{\>}{\rangle}
\newcommand{\inv}{^{-1}}

\renewcommand{\emph}{\textit}

\newcommand{\csps}{cor\-re\-spond\-en\-ces}
\newcommand{\csp}{cor\-re\-spond\-en\-ce}

\begin{document}

\title[$C^*$-algebras arising from topological dynamical systems]{The crossed-product structure of $C^*$-algebras arising from topological dynamical systems}
\author[Farthing, Patani, and Willis]{Cynthia Farthing, Nura Patani, and Paulette N. Willis}
\address [Cynthia Farthing]{Department of Mathematics, Creighton University, Omaha, Nebraska 68106}\email{CynthiaFarthing@creighton.edu}
\address [Nura Patani]{School of Mathematical and Statistical Sciences, Arizona State University, Tempe, Arizona 85287} \email{nura.patani@asu.edu}
\address [Paulette N. Willis]{Department of Mathematics, University of Houston, Houston, Texas 77204}\email{pnwillis@math.uh.edu}

\thanks{The third author was supported by the NSF Mathematical Sciences Postdoctoral Fellowship DMS-1004675, the University of Iowa Graduate College Fellowship as part of the Sloan Foundation Graduate Scholarship Program, and the University of Iowa Department of Mathematics NSF VIGRE grant DMS-0602242.}

\date{\today}

\begin{abstract}
We show that every topological $k$-graph constructed from a locally compact Hausdorff space $\Omega$ and a family of pairwise commuting local homeomorphisms on $\Omega$ satisfying a uniform boundedness condition on the cardinalities of inverse images may be realized as a semigroup crossed product in the sense of Larsen.
\end{abstract}


\maketitle

\section{Introduction}

In \cite{cun:internal}, Cuntz constructed the crossed product of a $C^*$-algebra $A$ by an endomorphism $\alpha$ as a corner in an ordinary group crossed product.  Since that time, there have been many efforts (see \cite{pas, stacey}, for example) to develop a theory of crossed products of $C^*$-algebras by single endomorphisms as well as by semigroups of endomorphism.  In \cite{ex_nlend}, Exel proposed a new definition for the crossed product of $A$ by $\alpha$ that depends not only on the pair $(A,\alpha)$ but also on the choice of a \emph{transfer operator} (i.e., a positive continuous linear map $L:A\to A$ satisfying $L(\alpha(a)b)=aL(b)$).  Exel shows that the Cuntz-Krieger algebra of a given $\{0,1\}$-matrix may be realized as the crossed product arising from the associated Markov sub-shift and a naturally defined transfer operator.

Extending Exel's construction to non-unital $C^*$-algebras, Brownlowe, Raeburn, and Vittadello in \cite{brv} model directed graph $C^*$-algebras as crossed products.  In particular, they show that if $E$ is a locally finite directed graph with no sources, then $C^*(E)\cong C_0(E^\infty)\rtimes_{\alpha,L}\N$ where $E^\infty$ is the infinite-path space of $E$ and $\alpha$ is the shift map on $E^\infty$.

In another extension of Exel's construction, Larsen (in \cite{larsen:crossed}) develops a theory of crossed products associated to dynamical systems $(A,S,\alpha,L)$ where $A$ is a (not necessarily unital) $C^*$-algebra, $S$ is an abelian semigroup with identity, $\alpha$ is an action of $S$ by endomorphisms, $L$ is an action of $S$ by transfer operators, and for all $s\in S$, the maps $\alpha_s,L_s$ are extendible to $M(A)$ in an appropriate sense.

Given a locally compact Hausdorff space $\O$ and a family $\{T_i\}_{i=1}^k$ of pairwise-commuting local homeomorphisms of $\O$, Yeend \cite{yeend07} described the construction of an associated topological $k$-graph $(\L(\O,\{T_i\}_{i=1}^k),d)$.  Motivated in part by the ideas in \cite{brv} described above, we show that if $\L=(\L(\O,\{T_i\}_{i=1}^k),d)$ is the topological $k$-graph constructed from the data {$(\O,\{T_i\}_{i=1}^k)$}, then $C^*(\L)$ has a crossed product structure in the sense of Larsen in \cite{larsen:crossed}.  

Given a general topological $k$-graph, it is not always the case than an associated graph $C^*$-algebra may be constructed.  We show that $\Lambda$ is compactly aligned, a condition that ensures $C^*(\Lambda)$ exists. This generalizes a result of Willis in \cite{willis} in which she essentially shows that the result holds when $k=2$, $\O$ is compact, and the maps $T_1,T_2$ $*$-commute.  In \cite{clsv}, the authors show that when $\L$ is a compactly aligned topological $k$-graph, the $C^*$-algebra $C^*(\L)$ constructed from the boundary path groupoid is isomorphic to the Cuntz-Pimsner algebra $\mathcal{NO}_{X^\L}$ where $X^\L$ is the topological $k$-graph correspondence associated to $\L$.  We show that the product system $X^{Lar}$ associated to the dynamical system (in the sense of Larsen) arising from the data $(\O,\{T_i\}_{i=1}^k)$ is isomorphic to the topological $k$-graph correspondence $X^\L$ so that the associated Cuntz-Pimsner algebras are isomorphic.

To show that $\mathcal{NO}_X$ is isomorphic to the Larsen crossed product, we show that certain notions of covariance agree for representations of the product systems $X^\Lambda$ and $X^{Lar}$.  Our isomorphism result then follows from the universal properties of the associated $C^*$-algebras.

Brownlowe has shown in \cite{brownlowe} that the $C^*$-algebra of a finitely-aligned discrete $k$-graph has a crossed product structure.  He has suggested that the Cuntz-Nica-Pimsner algebra $\NO_X$ should be used to define a general crossed product by a quasi-lattice ordered semigroup of partial endomorphisms and partially-defined transfer operators. The fact that, in our setting, $\mathcal{NO}_X$ is isomorphic to the Larsen crossed product supports his proposal.

The paper is organized as follows:  We begin with some preliminaries in Section~\ref{prelim}.  We state some necessary definitions about product systems of $C^*$-\csps, various notions of Cuntz-Pimsner covariance appearing in the literature, and the $C^*$-algebras that are universal for such representations.  We review several definitions about the topological $k$-graphs and the dynamical systems described by Larsen in \cite{larsen:crossed}, as well as the $C^*$-algebras associated to each of these constructions.

In Section~\ref{sec:TDS}, we define what we mean by topological dynamical system and describe a uniform boundedness condition that is key to our results.  We describe Yeend's construction of a topological $k$-graph from a topological dynamical system and show that this topological $k$-graph is always compactly aligned (in the sense of \cite[Definition 2.3]{yeend07}).  We then show how an Exel-Larsen system may be associated to a topological dynamical system satisfying our uniform boundedness condition.

In Section~\ref{sec:XLar}, we define two product systems over $\N^k$ of $C^*$-cor\-res\-pond\-ences: the topological $k$-graph correspondence $X^\Lambda$ and the product system $X^{Lar}$ associated to the Exel-Larsen system, and show that the two product systems are isomorphic. Finally, in Section~\ref{sec:algebras}, we show that the $C^*$-algebras associated to the topological $k$-graph and the Exel-Larsen system arising from a given topological dynamical system are isomorphic.

\section{Preliminaries}\label{prelim}

\subsection{Product systems of $C^*$-\csps}
In this subsection, we give some key definitions for product systems of $C^*$-\csps, many of which
may be found in \cite[Section 2]{sy}.  For more details on right-Hilbert $C^*$-modules and $C^*$-
\csps, we refer the reader to \cite{lan:hilbert,tfb}.

Given a $C^*$-algebra $A$ and a countable semigroup $S$ with identity $e$, a
\emph{product system over $S$ of $A$-\csps} is a semigroup $X$ equipped with a semigroup
homomorphism $p:X\to S$ such that $X_s:=p\inv(s)$ is an $A$-\csp~ for each $s\in S$, $X_e=A$
(viewed as an $A$-\csp), the multiplication in $X$ implements isomorphisms $\b_{s,t}:X_s
\otimes_A X_t\to X_{st}$ for $s,t\in S\setminus\{0\}$, and multiplication in $X$ by elements of $X_e=A$ induces maps $\b_
{s,e}:X_s\otimes_A X_e\to X_s$ and $\b_{e,s}:X_e\otimes_A X_s\to X_s$.  For each $s\in S$, $\b_
{s,e}$ is an isomorphism by \cite[Corollary 2.7]{tfb}.

For each $s\in S$ and $\xi,\eta\in X_s$, the operator $\Th_{\xi,\eta}:X_s\to X_s$ defined by $\Th_
{\xi,\eta}(\zeta):=\xi\cdot\<\eta,\zeta\>_A$ is adjointable with $\Th_{\xi,\eta}^*=\Th_{\eta,\xi}$.  The
space $\KK(X_s):=\overline{\spn}\{\Th_{\xi,\eta}:\xi,\eta\in X_s\}$ is a closed two-sided ideal in $\LL
(X_s)$ which we call the \emph{generalized compact operators on $X_s$}.

Given $s,t\in S$ with $s\neq e$, we have a homomorphism $\iota_s^{st}:\LL(X_s)\to \LL(X_{st})$
characterized by
	\[
	\iota_s^{st}(T)(\xi\eta)=T(\xi)\eta \all \xi\in X_s,\eta\in X_t, T\in \LL(X_s).
	\]
Via the identification of $\KK(X_e)$ with $A$, there is also a homomorphism $\iota_e^s:\KK(X_e)\to
\LL(X_s)$ given by $\iota_e^s=\phi_s$, where $\phi_s$ is the homomorphism of $A$ to $\LL(X_s)$
implementing the left action.

\subsection{Representations of product systems and associated $C^*$-algebras}
Given a product system $X$ over $S$ of $A$-\csps, a \emph{(Toeplitz) representation} of $X$ in a
$C^*$-algebra $B$ is a map $\psi:X\to B$ such that
\begin{enumerate}
	\item For each $s\in S$, the pair $(\psi_s,\psi_e):=(\psi|_{X_s},\psi|_{X_e})$ is a Toeplitz
representation of $X_s$ in the sense that $\psi_s:X_s\to B$ is linear and $\psi_e:A\to B$ is a
homomorphism satisfying
	\begin{align*}
		\psi_s(\xi\cdot a) &=\psi_s(\xi)\psi_e(a) \\
		\psi_s(\xi)^*\psi(\eta) &=\psi_e(\<\xi,\eta\>_{X_s}) \\
		\psi_s(a\cdot\xi) &=\psi_e(a)\psi_s(\xi)
	\end{align*}
	for $\xi, \eta\in X_s$, $a\in A$, and
	\item $\psi(\xi\eta)=\psi(\xi)\psi(\eta)$, for $\xi,\eta\in X$.
\end{enumerate}
For each $s\in S$, there is a homomorphism $\psi^{(s)}:\KK(X_s)\to B$ satisfying
\[
\psi^{(s)}(\Th_{\xi,\eta})=\psi_s(\xi)\psi_s(\eta)^* \for \xi,\eta\in X_s.
\]
We say that a representation $\psi:X\to B$ is \emph{Cuntz-Pimsner covariant} if for each $s\in S$
the (Toeplitz) representation $(\psi_s,\psi_e)$ is Cuntz-Pimsner covariant, that is
\begin{equation}\label{eqn:CP-K}
\psi^{(s)}(\phi_s(a))=\psi_e(a) \for a\in \phi_s^{-1}(\KK(X_s))\cap (\ker \phi_s)^\perp \tag{CP-K},
\end{equation}
where $\phi_s:A\to \LL(X_s)$ is the homomorphism giving the left action of $A$ on $X_s$.

\begin{rem}
Different definitions exist in the literature for Cuntz-Pimsner covariant representations.  The one
used above is sometimes referred to as the ``Katsura convention'' and differs from the definition in
\cite{fowmuhrae} where \eqref{eqn:CP-K} is instead required to hold for $a\in \phi^{-1}(\KK(X))$.  The two
definitions coincide when the left action on each fibre is injective.
\end{rem}

\begin{defn}\label{def:OX}
For a product system $X$ the \emph{Cuntz-Pimsner algebra $\OO_X$} is the universal $C^*$-algebra generated by a
representation
	$
	j^{Fow}:X\to \OO_X
	$
	that satisfies \eqref{eqn:CP-K}.
\end{defn}

\noindent A (Toeplitz) representation $\psi:X\to B$ is said to be \emph{coisometric on $K=\{K_s\}_{s
\in S}$}, where each $K_s$ is an ideal in $\phi_s^{-1}(\KK(X_s))$,  if each $(\psi_s,\psi_e)$ is
coisometric on $K_s$; that is,
	\begin{equation}\label{eqn:coisom}
	\psi^{(s)}(\phi_s(a))=\psi_e(a), \textrm{ for all }a\in K_s.
	\end{equation}

\begin{defn}\label{def:relCPalg}
The \emph{relative Cuntz-Pimsner algebra $\OO(X,K)$} is the universal $C^*$-algebra generated
by a representation
	$
	j^{relCP}:X\to\OO(X,K)
	$
	that is coisometric on $K=\{K_s\}_{s\in S}$.
\end{defn}

A quasi-lattice ordered group $(G,P)$ is a discrete group $G$ and a subsemigroup $P$ such that:
$P\cap P\inv=\{e\}$, and any two elements $p,q\in G$ that have a common upper bound in $P$ have a least upper bound
$p\vee q \in P$ under the order $p\leq q\Longleftrightarrow p\inv q\in P$.

\begin{defn}\label{def:ca}
Given a quasi-lattice ordered group and a product system $X$ over $P$ of $A$-\csps, we say that $X$ is \emph{compactly aligned} if
whenever $p\vee q<\infty$, the map $\iota_p^{p\vee q}(S)\iota_q^{p\vee q}(T)\in \KK(X_{p\vee q})$ for all $S\in\KK(X_p), T\in\KK(X_q)$.
	
\end{defn}
A (Toeplitz) representation $\psi:X\to B$ is \emph{Nica covariant} if, for each $p,q\in P$ and for all
$S\in\KK(X_p), \ T\in\KK(X_q)$, we have
\begin{equation}\label{eqn:N}
\psi^{(p)}(S)\psi^{(q)}(T)=\left\{\begin{array}{ll} \psi^{(p\vee q)}(\iota_p^{(p\vee q)}(S)\iota_q^{(p\vee
q)}(T)), & \textrm{ if }p\vee q<\infty \\ 0 & \textrm{ otherwise.} \end{array}\right. \tag{N}
\end{equation}

In \cite{sy}, Sims and Yeend introduced a new notion of Cuntz-Pimsner covariance for compactly aligned product systems.  In order to define their notion of Cuntz-Pimsner
covariance, we need to consider the space $\widetilde{X}$ which serves as a sort of ``boundary'' of
$X$ (see \cite[Remark 3.10]{sy}).

 Given a quasi-lattice ordered group $(G,P)$ and a product system $X$ over $P$ of $A$-\csps,
define $I_e=A$ and for $p\in P\setminus\{e\}$ define $I_p=\bigcap_{e<r\leq p}\ker(\phi_r)$.  Note
that $I_p$ is an ideal of $A$.  For $q\in P$, define
 \[
 \widetilde{X}_q=\bigoplus_{p\leq q} X_p\cdot I_{p^{-1}q}.
 \]
 Each $\widetilde{X}_q$ is an $A$-\csp~with left action implemented by $\widetilde{\phi}_q:A\to \LL
(\widetilde{X}_q)$ where $(\widetilde{\phi}_q(a)\xi)(p) = \phi_p(a)\xi(p)$, for $p\leq q$.
 There is a homomorphism $\widetilde{\iota}_p^q:\LL(X_p)\to \LL(\widetilde{X}_q)$   defined by
 \[
 \left(\widetilde{\iota}_p^q(S)\xi\right)(r)=\iota_p^r(S)\xi(r).
 \]

Let $(G,P)$ be a quasi-lattice ordered group and let $X$ be a compactly aligned product
system over $P$ of $A$-\csps~such that $\widetilde{\phi}_q$ is injective for each $q\in P$.  A
(Toeplitz) representation $\psi:X\to B$ of $X$ in a $C^*$-algebra $B$ is said to be \emph{Cuntz-Pimsner covariant} if
\begin{align} \label{eqn:CP-SY}
	\textrm{for every finite }F\subset P,\textrm{ and every choice }\{T_p\in \KK(X_p):p\in F\} \tag{CP-SY} \\
	\textrm{such that } \sum_{p\in F}\widetilde{\iota}_p^s(T_p)=0  \textrm{ for large $s$,  we have }
\sum_{p\in F}\psi^{(p)}(T_p)=0_B.\notag
\end{align}
See \cite[Definition 3.8]{sy} for the definition of \emph{for large $s$}.  If $\psi:X\to B$ satisfies both
\eqref{eqn:CP-SY} and \eqref{eqn:N}, then $\psi$ is said to be CNP-covariant or \emph{Cuntz-Nica-Pimsner covariant}.

\begin{defn}\label{def:NOX}
The \emph{Cuntz-Nica-Pimsner algebra $\NO_X$} is the universal $C^*$-algebra generated by a
CNP-covariant representation $j^{CNP}:X\to\NO_X.$
\end{defn}

\subsection{Topological $k$-graphs and their $C^*$-algebras}

For $k\in\N$, a \emph{topological $k$-graph} is a pair $(\L,d)$ consisting of: (1) a
small category $\L$ endowed with a second countable locally compact Hausdorff topology under
which composition is continuous and open, the range map $r$ is continuous, and the source map
$s$ is a local homeomorphism; and (2) a continuous functor $d:\L\to\N^k$, called the \emph
{degree map}, satisfying the \emph{factorization property}: if $\lambda\in\L$ with $d
(\lambda)=m+n$, then there are unique $\mu,\nu\in\L$ with $d(\mu)=m$, $d(\nu)=n$, and $
\lambda=\mu\nu$.  For $m\in\N^k$, let $\L^m$ denote the paths of degree $m$.  We identify $\L^0$ with the vertex space $\obj(\L)$.  For more details about topological $k$-graphs, see \cite{yeend07}.

Given a compactly aligned topological $k$-graph $\L$, the \emph{topological $k$-graph $C^*$-algebra} $C^*(\L)$ is the full groupoid $C^*$-algebra $C^*(\mathcal{G}_\L)$ of the boundary path groupoid $\mathcal{G}_\L$ defined in \cite[Definition 4.1]{y_gpoid}.  It is shown in \cite[Theorem 5.20]{clsv} that $C^*(\L)$ is isomorphic to the Cuntz-Nica-Pimsner algebra $\NO_{X^\L}$ associated to the topological $k$-graph cor\-re\-spond\-ence $X^\L$ (for details of the construction of $X^\L$, see \cite{sy} for example), where $\NO_{X^\L}$ is the universal $C^*$-algebra generated by a CNP-covariant representation $j^{CNP}:X^\L\to\NO_{X^\L}$.

\subsection{Exel-Larsen systems and their relative Cuntz-Pimsner algebras}
Let $A$ be a (not necessarily unital) $C^*$-algebra, $S$ an abelian semigroup with
identity $e$.  Let $\a:S\to \End(A)$ be an action such that each $\a_s$ is \emph{extendible,} meaning that it extends uniquely to an endomorphism $\overline{\a_s}$ of $M(A)$ such that
	\begin{equation}\label{eq:alphaext}
	\overline{\a}_s(1_{M(A)})=\lim \a_s(u_\lambda)
	\end{equation}
	for some (and hence every) approximate unit $(u_\lambda)$ in $A$ and all $s\in S$.  Finally,
let $L$ be an action of $S$ by continuous, linear, positive maps $L_s:A\to A$ which have linear
continuous extensions $\overline{L}_s:M(A)\to M(A)$ satisfying the \emph{transfer operator identity}
	\begin{equation}\label{eqn:transferop}
	L_s(\a_s(a)u)=a\overline{L}_s(u), \all a\in A, u\in M(A), s\in S.
	\end{equation}
	We call the quadruple $(A,S,\a,L)$ an \emph{Exel-Larsen system}.

Given an Exel-Larsen system $(A,S,\a,L)$ there is an associated product system over
$S$ of $A$-\csps~which we will denote $X^{Lar}$ (for details of the general construction, see \cite
{larsen:crossed}). The \emph{Larsen crossed product} $A\rtimes_{\a,L}S$ is the relative Cuntz-Pimsner algebra of $X^{Lar}$ and $K=\{K_s\}_{s\in S}$ where
	\begin{equation}\label{eqn:coisomsets}
		K_s=\overline{A\a_s(A)A}\cap\phi_s^{-1}(\KK(X^{Lar}_s)).
	\end{equation}
We denote by $j^{Lar}$ the universal representation of $X^{Lar}$ that generates $A\rtimes_{\a,L}
\N^k$.

\section{The topological dynamical system $(\O,\{T_i\}_{i=1}^k)$ and associated constructions}\label{sec:TDS}

\begin{defn}\label{TDS}
A \emph{topological dynamical system (TDS)} is a pair $(\O,\{T_i\}_{i=1}^k)$ consisting of a locally
compact Hausdorff space $\O$ and pairwise commuting local homeomorphisms $T_1,\ldots,T_k:\O
\to\O$.
For each $m=(m_1,\ldots,m_k)\in\N^k$ the map $\Th_m:\O\to\O$ defined by
	\[
	\Th_m(x)=T_1^{m_1}\cdots T_k^{m_k}(x).
	\]
is a local homeomorphism.
\end{defn}

\begin{defn}\label{def:unifboundinvimage}
Let $X$ and $Y$ be sets.  A function $f:X\to Y$ has \emph{uniformly bounded cardinality on inverse images} if there exists $N\in\N$ such that $\sup_{y\in Y}\left|\{x\in X:f(x)=y\}\right|\leq N$.   The number $N$ is called the uniform bound on the cardinality of the inverse image of $f$.

We say that a topological dynamical system $(\O, \{T_i\}_{i=1}^k)$ \emph{satisfies condition (UBC)} if each $T_i$, $1\leq i\leq k$, has uniformly bounded cardinalities on inverse images.
\end{defn}

If $(\O, \{T_i\}_{i=1}^k)$ satisfies condition (UBC), then for each $m\in\N^k$, the local homeomorphism $\Theta_m$ also has uniformly bounded cardinality on inverse images.

\begin{exs}\label{exs:tds}
\begin{enumerate}\renewcommand{\theenumi}{\ref*{exs:tds} (\roman{enumi})}

	\item\label{ex:nfold1} 
	Let $\T$ denote the unit circle and fix $n_0\in\N$.  Define $T:\T\to\T$ by $z\mapsto z^{n_0}$.  Then $
(\T,T)$ is a TDS and for $m\in \N$, the local homeomorphism $\Th_m$ is
given by $z\mapsto z^{n_0^m}$.  The system $(\T,T)$ has satisfies condition (UBC) since $$\sup_{y\in\T} |\{z\in \T:T(z)=z^{n_0}=y\}|\leq n_0.$$

	\item\label{ex:nfoldk} 
	Given any $n_1,\ldots,n_k\in\N$ we may define $T_i:\T\to\T$ by $z\mapsto z^{n_i}$ to obtain
$k$ pairwise commuting local homeomorphisms.  Then $(\T,\{T_i\}_{i=1}^k)$ is a topological
dynamical system and for each $m\in\N^k$, the local homeomorphism $\Th_m$ is given by $z
\mapsto z^{n_1^{m_1}+\cdots+n_k^{m_k}}$.  The system $(\T,\{T_i\}_{i=1}^k)$ satisfies condition (UBC) since $$\sup_{y\in\T} |\{z\in \T:T_i(z)=z^{n_i}=y\}|\leq n_i.$$

	\item\label{ex:fullshift} 
	Let $A$ be a finite alphabet and for $n\in\N$, let $A^n$ denote the space of words of length
$n$.  We let $A^\N$ denote the one-sided infinite sequence space, which is compact by
Tychonoff's Theorem.  The \emph{shift map} $\sigma:A^\N\to A^\N$ defined by
	\[
	\sigma(x_1x_2x_3\cdots)=x_2x_3\cdots
	\]
	is a local homeomorphism of $A^\N$.  Given a \emph{block map} $d:A^n\to A$ for some $n\in
\N$, we may define a \emph{sliding block code} $\tau_d:A^\N\to A^\N$ via
	\[
	\tau_d(x)_i=d(x_i\cdots x_{i+n-1}).
	\]
A function $\phi:A^\N\to A^\N$ is continuous and commutes with the shift map $\sigma$ if and only if $\phi= \tau_d$ is the sliding block code associated to some block map $d$ (see \cite[Lemma 3.3.3 and Lemma 3.3.7]{willis} for a proof, or \cite[Theorem 3.4]{Hed69} for an earlier proof in the two-sided setting).  Exel and Renault prove in \cite[Theorem 14.3]{exren} that $\tau_d$
is a local homeomorphism whenever $d$ is \emph{progressive} (also called \emph{right permutive}) in the sense that for each $x_1\cdots x_{n-1}\in A^{n-1}$ the function $a\mapsto d(x_1\cdots x_{n-1}a)$ is bijective. It follows that if $\tau_d$ is a sliding block code associated to a progressive block map, then $(A^\N,\{\sigma,\tau_d\})$ is a TDS and for $(a,b)\in\N^2$ the local homeomorphism $\Th_{(a,b)}$ is given by
	\[
	\Th_{(a,b)}(x)=\sigma^a\tau_d^b(x).
	\]
A block map $d$ is said to be regressive (also called \emph{left permutive}) if for each $x_1\cdots x_{n-1}\in A^{n-1}$ the function $a\mapsto d(ax_1\cdots x_{n-1})$ is bijective.  In \cite[Theorems 6.6 and 6.7]{Hed69}, Hedlund shows that if $\tau_d$ is a sliding block code associated to a block map that is both progressive and regressive, then $\tau_d$ is $|A|^{n-1}$-to-1 and surjective.  Therefore the system $(A^\N,\{\sigma,\tau_d\})$ satisfies condition (UBC) whenever the block map $d$ is progressive and regressive.

	\item\label{ex:kgraphshift} 
	Let $\L$ be a row-finite $k$-graph with no sources such that for each $i=1,\ldots,k$, $$|\Lambda^{e_i}v|<\infty, \all v\in\L^0.$$  Since $\L$ is row-finite with no sources, the
boundary path space $\partial\L$ coincides with the infinite path space $\L^\infty$ (see \cite
[Examples 5.13, 1.]{fmy}).  For each $i=1,\ldots,k$, let $T_i:\partial\L\to\partial\L$ denote the shift by
$e_i$, that is,
	\[
	T_i(x)(n)=x(n+e_i) \for n\in\N^k.
	\]
	Then $(\partial\L,\{T_i\}_{i=1}^k)$ is a topological dynamical system and for $m\in\N^k$, the
local homeomorphism $\Th_m$ is given by $\Th_m(x)(n)=x(n+m)$ for $n\in\N^k$.  It is straightforward to see that the condition $|s_{e_i}\inv(v)|<\infty$ for each $i=1,\ldots,k$ and every $v\in\L^0$ ensures the system $(\partial\L,\{T_i\}_{i=1}^k)$ satisfies condition (UBC).
	
\end{enumerate}
\end{exs}

\subsection{The topological $k$-graph associated to $(\O,\{T_i\}_{i=1}^k)$}\label{subsec:lambda}

Given a topological dynamical system $(\O,\{T_i\}_{i=1}^k)$ with local homeomorphisms $\Th_m$
as defined above, we construct a topological $k$-graph $\L=(\L(\O,\{T_i\}_{i=1}^k),d)$ as in \cite[Example 2.5
(iv)]{yeend07}.  Specifically, we have \begin{itemize}
	\item $\obj(\L)=\O$
	\item $\mor(\L)=\N^k\times\O$, with the product topology
	\item $r(n,x)=x$ and $s(n,x)=\Th_n(x)$
	\item Composition is given by $$(n,x)\circ(m,\Th_n(x))=(n+m,x)$$
	\item The degree map is defined by $d(n,x)=n$.
\end{itemize}

\begin{exs}\label{exs:graph}
\begin{enumerate}\renewcommand{\theenumi}{\ref*{exs:graph} (\roman{enumi})}
	\item \label{ex:nfold1:graph} 
	Fix $n_0\in \N$. Let $(\T,T)$ be the topological dynamical system described in \exref{ex:nfold1}.
The associated topological 1-graph is visualized below.
	
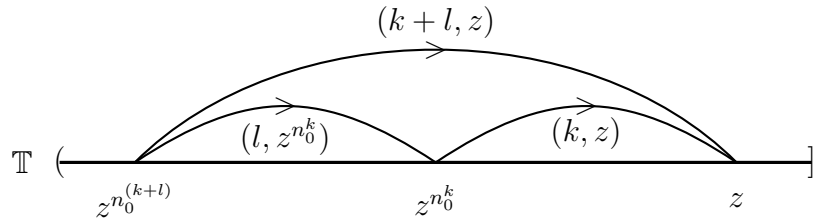
\begin{center}
	\begin{tikzpicture}
		\draw [very thick] (-5,0) -- (5,0);
		\node [very thick] (e) at (-5,0) {$($};
		\node [very thick] (f) at (5,0) {$]$};
		\node [very thick] (g) at (-5.5,0) {$\mathbb{T}$};
		\node (a) at (-4,-0.5) {$z^{n_0^{(k+l)}}$};
		\node (b) at (4,-0.5) {$z$};
		\draw [thick] (-4,0) .. controls (-2,2) and (2,2) .. node[sloped, above] {$(k+l,z)$} node {$>$}
(4,0);
		\node (c) at (0,-0.5) {$z^{n_0^k}$};
		\draw [thick] (0,0) .. controls (1.5,1) and (2.5,1) .. node[sloped,below] {$(k,z)$} node {$>$}
(4,0);
		\draw [thick]  (-4,0) .. controls (-2.5,1) and (-1.5,1) .. node[sloped,below] {$(l,z^{n_0^k})$}
node {$>$} (0,0);
	\end{tikzpicture}
	\end{center}
	\item \label{ex:nfoldk:graph} 
	For $n_1,n_2\in\N$, we obtain a topological dynamical system $(\T,\{T_1,T_2\})$ as in \exref
{ex:nfoldk} where $T_i:\T\to\T$ is given by $T_i(z)=z^{n_i}$.  The 1-skeleton of the associated 2-
graph is visualized below.
	
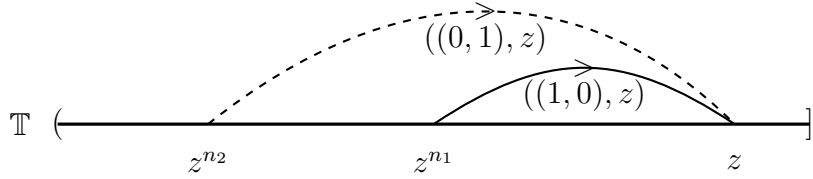
\begin{center}
	\begin{tikzpicture}
		\draw [very thick] (-5,0) -- (5,0);
		\node [very thick] (e) at (-5,0) {$($};
		\node [very thick] (f) at (5,0) {$]$};
		\node [very thick] (g) at (-5.5,0) {$\mathbb{T}$};
		\node (a) at (-3,-0.5) {$z^{n_2}$};
		\node (b) at (4,-0.5) {$z$};
		\node (c) at (0,-0.5) {$z^{n_1}$};
		\draw [thick] (0,0) .. controls (1.5,1) and (2.5,1) .. node[sloped,below] {$((1,0),z)$} node
{$>$} (4,0);
		\draw [thick, dashed]  (-3,0) .. controls (-0.5,2) and (2,2) .. node[sloped,below] {$((0,1),z)
$} node {$>$} (4,0);
	\end{tikzpicture}
	\end{center}
	\item \label{ex:kgraphshift:graph}
	Let $\L=\O_k$.  This is a locally-finite $k$-graph with no sources such that $|s_{e_i}\inv(v)|<\infty$  for each
$i=1,\ldots,k$ and every $\lambda$.  For $k=2$, the 1-skeleton of $\L$ is shown below:
	
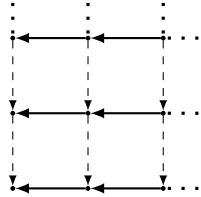
\begin{center}
	\begin{tikzpicture}[scale=1]
		 \node[inner sep=.5pt, circle, fill=black] (vert00) at (0,0) [draw] {};
		 \node[inner sep=.5pt, circle, fill=black] (vert10) at (1,0) [draw] {};
		 \node[inner sep=.5pt, circle, fill=black] (vert20) at (2,0) [draw] {};
		 \node[inner sep=.5pt, circle, color=white] (vert30) at (2.5,0) [draw] {};
		
		 \node[inner sep=.5pt, circle, fill=black] (vert01) at (0,1) [draw] {};
		 \node[inner sep=.5pt, circle, fill=black] (vert11) at (1,1) [draw] {};
		 \node[inner sep=.5pt, circle, fill=black] (vert21) at (2,1) [draw] {};
		 \node[inner sep=.5pt, circle, color=white] (vert31) at (2.5,1) [draw] {};
		
		 \node[inner sep=.5pt, circle, fill=black] (vert02) at (0,2) [draw] {};
		 \node[inner sep=.5pt, circle, fill=black] (vert12) at (1,2) [draw] {};
		 \node[inner sep=.5pt, circle, fill=black] (vert22) at (2,2) [draw] {};
		 \node[inner sep=.5pt, circle, color=white] (vert32) at (2.5,2) [draw] {};
		
		 \node[inner sep=.5pt, circle, color=white] (vert03) at (0,2.5) [draw] {};
		 \node[inner sep=.5pt, circle, color=white] (vert13) at (1,2.5) [draw] {};
		 \node[inner sep=.5pt, circle, color=white] (vert23) at (2,2.5) [draw] {};
		 \node[inner sep=.5pt, circle, color=white] (vert33) at (2.5,2.5) [draw] {};
		
		 \draw[style=thick, -latex] (vert12.west)--(vert02.east);
		 \draw[style=thick, -latex] (vert11.west)--(vert01.east);
		 \draw[style=thick, -latex] (vert10.west)--(vert00.east);
		 \draw[style=thick, -latex] (vert22.west)--(vert12.east);
		 \draw[style=thick, -latex] (vert21.west)--(vert11.east);
		 \draw[style=thick, -latex] (vert20.west)--(vert10.east);
		 \draw[style=loosely dotted, very thick, -] (vert32.west)--(vert22.east);
		 \draw[style=loosely dotted, very thick, -] (vert31.west)--(vert21.east);
		 \draw[style=loosely dotted, very thick, -] (vert30.west)--(vert20.east);
		
		 \draw[style=loosely dotted, very thick, -] (vert03.south)--(vert02.north);
		 \draw[style=loosely dotted, very thick, -] (vert13.south)--(vert12.north);
		 \draw[style=loosely dotted, very thick, -] (vert23.south)--(vert22.north);
		 \draw[style=dashed, -latex] (vert02.south)--(vert01.north);
		 \draw[style=dashed, -latex] (vert12.south)--(vert11.north);
		 \draw[style=dashed, -latex] (vert22.south)--(vert21.north);
		 \draw[style=dashed, -latex] (vert01.south)--(vert00.north);
		 \draw[style=dashed, -latex] (vert11.south)--(vert10.north);
		 \draw[style=dashed, -latex] (vert21.south)--(vert20.north);
		  
	\end{tikzpicture}
	\end{center}

Let $(\partial\L,\{T_i\}^k)$ be the topological dynamical system as in \exref{ex:kgraphshift}.
To visualize the associated topological $k$-graph $\Gamma$, note that, since $\L=\O_k$, each $x
\in \L^\infty$ is uniquely determined by a point $p=r(x)=x(0)\in\obj(\L)$.  So we may regard $\L^\infty
$ as $\N^k$.  Modifying notation to reflect this gives
\begin{itemize}
	\item $\obj(\Gamma)=\N^k$
	\item $\mor(\Gamma)=\N^k\times\N^k$
	\item $r(m,p)=p$ and $s(m,p)=p+m$
	\item $d(m,p)=m$
	\item $(m,p)(n,p+m)=(m+n,p)$
\end{itemize}
The 1-skeleton of $\Gamma$ is
	
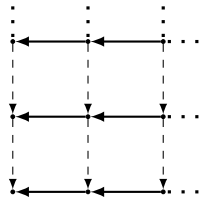
\begin{center}
	\begin{tikzpicture}[scale=1]
		 \node[inner sep=.5pt, circle, fill=black] (vert00) at (0,0) [draw] {};
		 \node[inner sep=.5pt, circle, fill=black] (vert10) at (1,0) [draw] {};
		 \node[inner sep=.5pt, circle, fill=black] (vert20) at (2,0) [draw] {};
		 \node[inner sep=.5pt, circle, color=white] (vert30) at (2.5,0) [draw] {};
		
		 \node[inner sep=.5pt, circle, fill=black] (vert01) at (0,1) [draw] {};
		 \node[inner sep=.5pt, circle, fill=black] (vert11) at (1,1) [draw] {};
		 \node[inner sep=.5pt, circle, fill=black] (vert21) at (2,1) [draw] {};
		 \node[inner sep=.5pt, circle, color=white] (vert31) at (2.5,1) [draw] {};
		
		 \node[inner sep=.5pt, circle, fill=black] (vert02) at (0,2) [draw] {};
		 \node[inner sep=.5pt, circle, fill=black] (vert12) at (1,2) [draw] {};
		 \node[inner sep=.5pt, circle, fill=black] (vert22) at (2,2) [draw] {};
		 \node[inner sep=.5pt, circle, color=white] (vert32) at (2.5,2) [draw] {};
		
		 \node[inner sep=.5pt, circle, color=white] (vert03) at (0,2.5) [draw] {};
		 \node[inner sep=.5pt, circle, color=white] (vert13) at (1,2.5) [draw] {};
		 \node[inner sep=.5pt, circle, color=white] (vert23) at (2,2.5) [draw] {};
		 \node[inner sep=.5pt, circle, color=white] (vert33) at (2.5,2.5) [draw] {};
		
		 \draw[style=thick, -latex] (vert12.west)--(vert02.east);
		 \draw[style=thick, -latex] (vert11.west)--(vert01.east);
		 \draw[style=thick, -latex] (vert10.west)--(vert00.east);
		 \draw[style=thick, -latex] (vert22.west)--(vert12.east);
		 \draw[style=thick, -latex] (vert21.west)--(vert11.east);
		 \draw[style=thick, -latex] (vert20.west)--(vert10.east);
		 \draw[style=loosely dotted, very thick, -] (vert32.west)--(vert22.east);
		 \draw[style=loosely dotted, very thick, -] (vert31.west)--(vert21.east);
		 \draw[style=loosely dotted, very thick, -] (vert30.west)--(vert20.east);
		
		 \draw[style=loosely dotted, very thick,  -] (vert03.south)--(vert02.north);
		 \draw[style=loosely dotted, very thick,  -] (vert13.south)--(vert12.north);
		 \draw[style=loosely dotted, very thick,  -] (vert23.south)--(vert22.north);
		 \draw[style=dashed,  -latex] (vert02.south)--(vert01.north);
		 \draw[style=dashed,  -latex] (vert12.south)--(vert11.north);
		 \draw[style=dashed,  -latex] (vert22.south)--(vert21.north);
		 \draw[style=dashed,  -latex] (vert01.south)--(vert00.north);
		 \draw[style=dashed,  -latex] (vert11.south)--(vert10.north);
		 \draw[style=dashed,  -latex] (vert21.south)--(vert20.north);
	\end{tikzpicture}
	\end{center}
\end{enumerate}
\end{exs}

Given a general topological dynamical system $(\O,\{T_i\}_{i=1}^k)$, it is important to verify that we
may in fact construct the topological $k$-graph $C^*$-algebra $C^*(\L)$ associated to the
topological $k$-graph $\L=(\L(\O,\{T_i\}_{i=1}^k),d)$.  In order to establish this, we begin by showing that $\L$ is proper.  We then prove that every proper topological $k$-graph is compactly aligned.

\begin{defn}\label{def:proper}
	A topological $k$-graph $\L$ is said to be \emph{proper} if for all $m\in\N^k$, the map $r|_
{\L^m}$ is a proper map.  That is, if for every $m\in\N^k$ and compact $U\subset\L^0$, the set $U
\L^m$ is compact.
\end{defn}

\begin{defn}\label{def:ca}
	A topological $k$-graph $\L$ is said to be \emph{compactly aligned} if for all $p,q\in \N^k$
and for all compact $U\subset \L^p$ and $V\subset\L^q$, the set $U\vee V:=U\L^{(p\vee q)-p}\cap V\L^{(p\vee q)-q} \subset\L^{p\vee q}$ is
compact.
\end{defn}
This compactly aligned condition ensures that the boundary path groupoid $G_\L$ is a locally
compact $r$-discrete groupoid admitting a Haar system and hence that the associated $C^*$-
algebra $C^*(\L)$ may be defined.

\begin{lem}\label{lem:lambdaisproper}
	Let $(\O, \{T_i\}^k_{i=1})$ be a topological dynamical system.  Then the associated topological
$k$-graph $\L=(\L(\O,\{T_i\}_{i=1}^k),d)$ is proper.
\end{lem}
\begin{proof}
	Fix $m\in\N^k$ and let $U\subset \L^0$ be compact.  We want to show that $U\L^m$ is
compact.  Let $\mathcal{C}=\{C_i\}_{i\in I}$ be an open cover of $U\L^m$.  Note that $U\L^m=\{m\}
\times U$ so for each $i$, we have $C_i=\{m\}\times B_i$ for some open $B_i\subseteq U$.  Then $
\mathcal{B}=\{B_i\}_{i\in I}$ is an open cover of $U$.
	
	Since $U$ is compact, there is a finite subset $J$ of $I$ such that $U\subseteq\bigcup_{i\in J}
B_i$.  Then
	\[
	U\L^m=\{m\}\times U\subseteq\{m\}\times\bigcup_{i\in J} B_i=\bigcup_{i\in J}\left(\{m\}\times B_i
\right)=\bigcup_{i\in J}C_i
	\]
	Then $\mathcal{C}'=\{C_i\}_{i\in J}$ is a finite subset of $\mathcal{C}$ that covers $U\L^m$.  It
follows that $U\L^m$ is compact and hence $\L$ is proper.
\end{proof}

To show that $\L=(\L(\O,\{T_i\}_{i=1}^k),d)$ is compactly aligned, we use the following lemma which is stated
without proof in \cite[Remark 6.5]{yeend07}.
\begin{lem}\label{lem:properiscompactlyaligned}
	Every proper topological $k$-graph is compactly aligned.
\end{lem}

\begin{proof}
	Let $p,q\in\N^k$ and let $U\subset\L^p$ and $V\subseteq\L^q$ be compact.  Since $s$ is
continuous, $s(U)$ and $s(V)$ are both compact.  By the assumption that $\Lambda$ is proper, it follows that $s(U)\Lambda^{(p\vee q)-p}$ and $s(V)\Lambda^{(p\vee q)-q}$ are both compact.  Moreover, since $\L*\L$ has the relative topology inherited from the product topology on $\L\times\L$, the sets $U*s(U)\L^{(p\vee q)-p}$ and $V*s(V)\L^{(p\vee q)-q}$ are compact.  Since the composition map is continuous, it follows that theimages of these sets under the composition map, namely $U\L^{(p\vee q)-p}$ and $V\L^{(p\vee q)-q}$, are compact. Hence
	\[
	U\vee V=U\L^{(p\vee q)-p}\cap V\L^{(p\vee q)-q}
	\]
	is compact and therefore $\L$ is compactly aligned.
\end{proof}

\begin{prop}\label{prop:lambdaca}
	Let $(\O, \{T_i\}^k_{i=1})$ be a topological dynamical system.  Then the associated topological
$k$-graph $\L=(\L(\O,\{T_i\}_{i=1}^k),d)$ is compactly aligned.
\end{prop}

\begin{proof}
This follows directly from \lemref{lem:lambdaisproper} and Lemma~\ref
{lem:properiscompactlyaligned}.
\end{proof}

\subsection{The Exel-Larsen system associated to $(\O,\{T_i\}_{i=1}^k)$}\label{subsec:exlar}

Given a topological dynamical system $(\O,\{T_i\}_{i=1}^k)$ that satisfies condition (UBC), we may construct an \emph{Exel-Larsen system}.  

\begin{lem}\label{lem:FvsK}
Let $(\Omega,\{T_i\}_{i=1}^k)$ be a topological dynamical system that satisfies condition (UBC).  For $m\in\N^k$, define $\a_m\in\End(C_0(\O))$ by 
	\[
	\a_m(f)=f \circ\Th_m, \for f\in C_0(\O)
	\]
so that $\a$ is an action of $\N^k$ on $C_0(\O)$. Then $\alpha$ extends uniquely to an endomorphism $\overline{\alpha_s}$ of $C_b(\Omega)$ satisfying \eqref{eq:alphaext}.
\end{lem}

\begin{proof}
First note that each $\a_m$ is nondegenerate, i.e., $\a_m(C_0(\O))C_0(\O)=C_0(\O)$.  To see this,  it is enough to show that for $g\in C_c(\O)$ there is $f\in C_c(\O)$ such that $\a(f)g=g$.  Since $\Th_m$ is continuous and $g$ is compactly supported, the set $\Th_m(\supp(g))$ is compact.  By Urysohn's Lemma for locally compact Hausdorff spaces, we may choose $f\in C_c(\O)$ such that $f|_{\Th_m(\supp(g))}=1$.  It follows that $\a_m(f)g=g$.  

Since $\a_m(C_0(\O))C_0(\O)=C_0(\O)$, the unique strictly continuous extension $\overline{\a}_m$ defined by $\overline{\a}_m(f)=f\circ\Th_m$, for $f\in C_b(\O)$ is unital (see \cite[Proposition 1.1.13]{jt}) so that $\overline{\a}_m$ satisfies \eqref{eq:alphaext} as desired.
\end{proof}

\begin{lem}\label{lem:transferopext}
	Let $(\O,\{T_i\}_{i=1}^k)$ be a topological dynamical system that satisfies condition (UBC).  For $m\in\N^k$, $f\in C_0(\O)$, and $x\in \O$ define ${L}_m$ by
	\[
	{L}_m(f)(x)=\begin{cases} \sum_{\Th_m(y)=x}f(y) & \textrm{if } x\in \Th_m(\O) \\ 0 & \textrm{else}
\end{cases}
	\]
	and similarly define $\overline{L}_m$ for $f\in C_b(\O)$.  Then each $L_m$ is a continuous,
linear, positive map on $C_0(\O)$ with continuous linear extension $\overline{L}_m$ satisfying
	\[
	L_m(\a_m(f)g)=f\overline{L}_m(g).
	\]
\end{lem}
\begin{proof}

Fix $m\in \N^k$.  To see that $L_m$ maps $C_0(\O)$ into $C_0(\O)$, let $x\in \O$.  Then there is an
open neighborhood $V$ of $x$ and an open neighborhood $U_y$ for each $y\in \Th_m^{-1}(\{x\})$
such that $\Th_m|_{U_y}:U_y\to V$ is a homeomorphism.  {Condition (UBC)} ensures that there are finitely many such sets $U_y$.  It follows then that $L_m (f)$ is the finite sum of the functions $f|_{U_y}$ and is hence is in $C_0(\O)$.
\color{black}

It is straightforward to see that, for each $m\in\N^k$, $L_m$ is continuous, linear, and positive.  We
must show that $\overline{L}_m$ is a continuous linear extension of $L_m$ satisfying \eqref
{eqn:transferop}.  By an argument similar to the one above, we know that $\overline{L}_m$ maps
$C_b(\O)$ to $C_b(\O)$.  Suppose $f,g\in C_b
(\O)$, $a,b\in\C$, $x\in\Th_m(\O)$.  Then
	\begin{align*}
		(a\overline{L}_m(f)-b\overline{L}_m(g))(x) &= a\overline{L}_m(f)(x)-b\overline{L}_m(g)(x) \\
		&= \sum_{\Th_m(y)=x}af(y)-\sum_{\Th_m(y)=x}bg(y) \\
		&= \sum_{\Th_m(y)=x}(af-bg)(y) \\
		&= \overline{L}_m(af-by) (x)
	\end{align*}
	If $x \notin \Th_m(\O)$, then both $(a\overline{L}_m(f)-b\overline{L}_m(g))(x)$ and $\overline
{L}_m(af-by) (x)$ are zero.  Thus, $\overline{L}_m$ is linear.
	Continuity of $\overline{L}_m$ follows from the fact that $(\O, \{T_i\}_{i=1}^k)$ {satisfies condition (UBC)} because for any $x\in \Th_m(\O)$, we have	
    \begin{align*}
		\left|\overline{L}_m(f)(x)\right| & = \left| \sum_{\Th_m(y)=x}f(y)\right| \\
		&\leq \sum_{\Th_m(y)=x}\left|f(y)\right| \\
		&\leq N_m\cdot \|f\|_\infty,
	\end{align*}
	where $N_m\in\N$ is the uniform bound on the cardinality of the inverse image of $\Th_m$.  If
$x\notin\Th_m(\O)$, then $\left|\overline{L}_m(f)(x)\right| =0$, and the inequality
	\[
	\left|\overline{L}_m(f)(x)\right|  \leq N_m\|f\|_\infty
	\]
	holds for all $x\in \O$.
	Taking the supremum over all $x\in \O$ gives that
	\[
	\|\overline{L}_m(f)\|\leq N_m\|f\|_\infty
	\]
	so that $\overline{L}_m$ is bounded.  Since $\overline{L}_m$ is a linear map on a normed
space, it follows that it is continuous.
	
Now if $f\in C_0(\O)$, $g\in C_b(\O)$, and $x\in \Th_m(\O)$, it follows that  \begin{align*}
		L_m(\a(f)g)(x) &= \sum_{\Th_m(y)=x} (\a_m(f)g)(y) \\
		&= \sum_{\Th_m(y)=x}f(\Th_m(y))g(y) \\
		&= f(x) \sum_{\Th_m(y)=x} g(y) \\
		&= (f\overline{L}_m(g))(x).
	\end{align*}
	When $x\notin\Th_m(\O)$, both $L_m(\a(f)g)(x)$ and $(f\overline{L}_m(g))(x)$ are zero.  Thus
$L_m(\a_m(f)g)=f\overline{L}_m(g).$
\end{proof}
It follows from Lemma~\ref{lem:FvsK} and Lemma~\ref{lem:transferopext} that $(C_0(\O),\N^k,\a,L)$ is an Exel-Larsen system.

\begin{exs}\label{exs:EL}
\begin{enumerate}\renewcommand{\theenumi}{\ref*{exs:EL} (\roman{enumi})}
	\item \label{ex:fullshift:EL}
	Let $A=\{0,1,2,3\}$ and define $d:A^2\to A$ via $(a,b)\mapsto a+b \mod 4$.  It is
straightforward to see that $d$ is both progressive and regressive and hence the associated sliding
block code $\tau_d$ is a local homeomorphism that has uniformly bounded cardinalities on inverse
images.  Then $(A^\N,\{\sigma,\tau_d\})$ is a topological dynamical system {that satisfies condition (UBC)} as in \exref{ex:fullshift}.
Define $\a:\N^2\to C(A^\N)$ by $\a_{(m,n)}(f)=f\circ\Th_{(m,n)}$ and, since each $\Th_m$ is
surjective, define $L_{(m,n)}:C(A^\N)\to C(A^\N)$ by
\[
L_{(m,n)}(f)(x)=\sum_{\Th_{(m,n)}(y)=x}f(y)
\]
for $(m,n)\in\N^2$, the quadruple $(A^\N,\N^2, \a,L)$ is an Exel-Larsen system.
	\item \label{ex:kgraphshift:EL}
	Let $\L=\O_k$ and let $(\partial\L,\{T_i\}_{i=1}^k)$ be the topological dynamical system as in
\exref{ex:kgraphshift}.  We may define the Exel-Larsen system $(C_0(\L^\infty),\N^k,\a,L)$ where
	\begin{align*}
		\a_m(f)(x) &= f\circ\sigma_m (x) \\
		L_m(f)(x) &= \begin{cases} \sum_{\sigma_m(y)=x} f(y) & \textrm{ if } x\in \sigma_m(\L^
\infty) \\ 0 & \textrm{ else } \end{cases}
	\end{align*}
	Again regarding $\L^\infty$ as $\N^k$, it is straightforward to show that
	\begin{align*}
		\a_m(f)(n) &= f(n+m) \\
		L_m(f)(n) &= \begin{cases} f(n-m) & \textrm{ if } n-m\in \N^k \\ 0 & \textrm{ else} \end
{cases}
	\end{align*}
	so that we obtain the Exel-Larsen system $(C_0(\N^k),\N^k,\a,L)$.
\end{enumerate}
\end{exs}

\section{The associated product systems}\label{sec:XLar}
Associated to the topological $k$-graph $\L=(\L(\O,\{T_i\}_{i=1}^k),d)$ and the Exel-Larsen system $(C_0(\O),
\N^k,\a,L)$ are product systems over $\N^k$ of $C_0(\O)$-\csps, denoted $X^{\L}$ and
$X^{Lar}$ respectively.  We show in Theorem \ref{prop:prodsysisom} that the two product
systems are in fact isomorphic.

\begin{defn}\label{def:topkgraphcorresp}
	The \emph{topological $k$-graph \csp} associated to a topological $k$-graph $
\L$ is the product system $X^\L$ over $\N^k$ of $C_0(\O)$-\csps~such that:	\begin
{enumerate}
		\item For each $m\in\N^k$, $X^\L_m$ is the topological graph \csp~associated to the
topological graph
	\begin{align*}
		E_m &= (\L^0,\L^m,r|_{\L^m},s|_{\L^m}) \\
		&= (\O,\{m\}\times\O,r_m,s_m)
	\end{align*}
	In particular, $X^\L_m$ is a completion of $C_c(\{m\}\times\O)$ (see \cite{katsgen} for details)
and the $C_0(\O)$-bimodule operations and $C_0(\O)$-valued inner product are given by
	\begin{align*}
		(f\cdot\xi\cdot g)(m,x) &= f(r(m,x))\xi(m,x)g(s(m,x)) \\
		&= f(x)\xi(m,x)g(\Th_m(x)), \textrm{ and} \\
		\<\xi,\eta\>_m(x) &= \sum_{(m,y)\in s_m^{-1}(x)}\overline{\xi(m,y)}\eta(m,y).
	\end{align*}
	
		\item For $m,n\in\N^k$, $\beta^\L_{m,n}:X^\L_m\otimes_{C_0(\O)}X^\L_n\to X^\L_{m+n}$ is
defined by
	\[
	\beta^\L_{m,n}(\xi\otimes\eta)(m+n,x)=\xi(m,x)\eta(n,\Th_m(x)).
	\]
	\end{enumerate}
\end{defn}

{\begin{rem}
It is important to note that a topological $k$-graph correspondence is not a $C^*$-correspondence, but is instead a product system over $\N^k$ of $C^*$-correspondences.  We use the terminology \emph{topological $k$-graph correspondence} to agree with the existing notions of \emph{graph correspondence} (see \cite{raeburngraph} for example) and \emph{topological graph correspondence} (see \cite{katsgen}).  The Cuntz-Pimsner algebras of a graph correspondence and a topological graph correspondence are isomorphic to the graph $C^*$-algebra and topological graph $C^*$-algebra, respectively.  Similarly, a generalization of the Cuntz-Pimsner algebra of the topological $k$-graph correspondence is isomorphic to the topological $k$-graph $C^*$-algebra (see \cite[Theorem 5.20]{clsv}).
\end{rem}}

\begin{defn}\label{def:XLar}
	The \emph{product system associated to the Exel-Larsen system $(C_0(\O),\N^k,\a,L)$} is
the product system $X^{Lar}$ over $\N^k$ of $C_0(\O)$-\csps~such that:
	\begin{enumerate}
		\item For each $m\in\N^k$, $X^{Lar}_m=\{m\}\times C_0(\O)$ with $C_0(\O)$-bimodule
operations
	\[
	f\cdot(m,g)\cdot h=(m,fg\a_m(h)),
	\]
	where $(fg\a_m(h))(x)=f(x)g(x)h(\Th_m(x))$, and $C_0(\O)$-valued inner product
	\[
	\<(m,f),(n,g)\>_m(x)=L_m(f^*g)(x)=\begin{cases} \sum_{\Th_m(y)=x}\overline{f(y)}g(y) & \textrm
{if } x\in\Th_m(\O) \\ 0& \textrm{otherwise}.\end{cases}
	\]
		\item  For $m,n\in\N^k$, $\beta^{Lar}_{m,n}:X^{Lar}_m\otimes_{C_0(\O)}X^{Lar}_n\to X^
{Lar}_{m+n}$ is defined by
	\[
	\beta^{Lar}_{m,n}((m,f)\otimes(n,g))=(m+n,f\a_m(g)).
	\]
	\end{enumerate}
\end{defn}

\subsection{The isomorphism of product systems}

\begin{thm}\label{prop:prodsysisom}
	Let $(\O, \{T_i\}^k_{i=1})$ be a topological dynamical system {satisfying condition (UBC)}, let $\L=(\L(\O,\{T_i\}_{i=1}^k),d)$ be the
associated topological $k$-graph, and let $(C_0(\O),\N^k,\a,L)$ be the associated Exel-Larsen
system.  Then the topological $k$-graph \csp~$X^{\L}$ is isomorphic to the product system $X^
{Lar}$ associated to the Exel-Larsen system.
\end{thm}

To show that the product systems are isomorphic we must show there is a map $\psi:X^{Lar}\to X^\L$ satisfying
	\begin{enumerate}
		\item for each $m\in\N^k$, the map $\psi_m=\psi|_{X^{Lar}_m}:X^{Lar}_m\to X^\L_m$ is
a $C_0(\O)$-\csp~isomorphism that preserves inner product, and
		\item $\psi$ respects the multiplication in the semigroups $X^{Lar}$ and $X^\L$.
	\end{enumerate}

The following lemmas are helpful in proving our result.
\begin{lem}\label{lem:bound}
	Let $E=(E^0,E^1,r,s)$ be a topological graph such that the source map $s:E^1\to E^0$ has
uniformly bounded cardinalities on inverse images.  Then the associated graph correspondence $X_E=C_0(E^1)$ as an algebraic $C_0(E^0)$-bimodule.
\end{lem}
\begin{proof}
By \cite[Lemma 1.6]{katsgen}, $C_c(E^1)$ is dense in $X_E$, hence $X_E\subseteq C_0(E^1)$.  For the reverse containment, let $\xi\in C_0(E^1)$.  Since $s:E^1\to E^0$ has uniformly bounded cardinalities on inverse images, there is $M\in\N$
such that $|\{e\in E^1:s(e)=v\}|\leq M$ for every $v\in E^0$.

 To see that the map $v\mapsto \sum_{e\in s^{-1}(v)}|\xi(e)|^2$ is in $C_0(E^0)$, note that for $v\in
\O$ there is a neighborhood $V$ of $v$ and finitely many open sets $U_{e_i}$ such that $s$
restricts to a homeomorphism from $U_{e_i}$ onto $V$. It follows that $v\mapsto \sum_{e\in s^{-1}
(v)}|\xi(e)|^2$ is a finite sum of continuous functions and is therefore in $C_0(E^0)$ and hence $\xi
\in X_E$.
\end{proof}

\begin{lem}\label{lem:prodsysisom}
	Fix $m\in \N^k$.  For each $f\in C_0(\O)$, the function $\widetilde{f}:\{m\}\times \O\to\C$
defined by
	\[
	\widetilde{f}(m,x)=f(x)
	\]
	is an element of $X_m^\L$.
\end{lem}
\begin{proof}
Since $X_m^\L=X_{E^\L_m}$, by Lemma~\ref{lem:bound} it is sufficient to show that $\widetilde{f}
\in C_0(E_m^1)$ where $E_m^1=\{m\}\times \O$.  Note that $\widetilde{f}$ is the composition of $f$
with the homeomorphism $(m,x)\mapsto x$ of $\{m\}\times\O$ onto $\O$.  Then $\widetilde{f}$ is
continuous since $f$ is. If $\epsilon>0$ and $K$ is a compact set such that $|f(x)|\leq \epsilon$ for
$x\in \O\setminus K$, then $|\widetilde{f}(m,x)|<\epsilon$ for $(m,x)\in E_m^1 \setminus (\{m\}\times
K)$.  Hence $\widetilde{f}\in X^\L_m$ as desired.
\color{black}
\end{proof}
 By a similar argument, we see that for each $\xi\in X^\L_m$, the function $\hat{\xi}=(m,\eta)$
where
\[
\eta(x)=\xi(m,x)
\]
is an element of $X^{Lar}_m$.

We define $\psi:X^{Lar}\to X^\L$ by letting $\psi_m:X^{Lar}_m\to X^\L_m$ be given by
\[
\psi_m(m,f)=\widetilde{f}
\]
for each $m\in\N^k$.

\begin{proof}[Proof of Theorem~\ref{prop:prodsysisom}]
	Fix $m\in\N^k$.  That $\psi_m:X^{Lar}_m \to X^\L_m$ is a $C_0(\O)$-cor\-re\-spond\-ence
isomorphism preserving the inner product follows directly from the fact that the map $(m,x)\mapsto
x$ is a homeomorphism of $\{m\}\times\O$ onto $\O$.  To see that $\psi$ respects the semigroup
multiplication, let $(m,f)\in X^{Lar}_m, (n,g)\in X^{Lar}_n, (m+n,x)\in \{m+n\}\times\O$.  Then
	\begin{align*}
		\left(\psi_m(m,f)\psi_n(n,g)\right)(m+n,x) &= \psi_m(m,f)(m,x)\psi_n(n,g)(n,\Th_m(x)) \\
		&= \widetilde{f}(m,x)\widetilde{g}(n,\Th_m(x)) \\
		&= f(x) g(\Th_m(x)) \\
		&= f\a_m(g) (x) \\
		&= \psi_{m+n}(m+n,f\a_m(g))(m+n,x).
	\end{align*}
	Hence $\psi_m(m,f)\psi_n(n,g)=\psi_{n+m}(n+m,f\a_m(g))$ as desired.
\end{proof}

\begin{rem} \label{rem:ca}
We showed in \propref{prop:lambdaca} that the topological $k$-graph $\L=(\L(\O,\{T_i\}_{i=1}^k),d)$ is
compactly aligned.  By \cite[Proposition 5.15]{clsv}, this happens if and only if $X^\L$ is compactly
aligned in the sense of Definition~\ref{def:ca}.  In \cite{willis}, Willis shows that if $k=2$, $\O$ is
compact, and $T_1$ and $T_2$ $*$-commute (in the sense that whenever $T_1(x)=T_2(y)$, there
is a unique $z\in\O$ with $T_1(z)=y$ and $T_2(z)=x$) then the product system $X^{Lar}$
constructed from the Exel-Larsen system $(C(\O),\N^2,\a,L)$ is compactly aligned.  \thmref
{prop:prodsysisom} together with \propref{prop:lambdaca} then imply that Willis' result is true for
arbitrary $k\in\N$, locally compact $\O$ and that the $*$-commuting restriction may be lifted.
\end{rem}

\section{The Larsen crossed product and $C^*(\L)$}\label{sec:algebras}

In this section, we show that the $C^*$-algebras $C^*(\L)$ and $C_0(\O)\rtimes_{\a,L}\N^k$
associated to the topological $k$-graph $\L=(\L(\O,\{T_i\}_{i=1}^k),d)$ and the Exel-Larsen system $(C_0(\O),
\N^k,\a,L)$, respectively, are isomorphic.  In particular, since $C^*(\L)\cong \NO_{X^\L}$, we show
that $C_0(\O)\rtimes_{\a,L}\N^k\cong \NO_{X^\L}$ and hence that
\[
C^*(\L) \cong C_0(\O)\rtimes_{\a,L}\N^k.
\]

\begin{rem}
	In \cite[Example 7.1(iii)]{yeend07}, Yeend describes the associated topological $k$-graph
$C^*$-algebra in the case where the maps $\{T_i\}^k_{i=1}$ are homeomorphisms.  Surjectivity of the maps ensures that the associated topological $k$-graph $\L$ has no sources.  As a result, the boundary path groupoid is amenable.  Since the maps are homeomorphisms, there is an induced action $\a$ of $\Z^k$ on $C_0(\O)$ defined by
	\[
	\a_m(f)(x)=f(\Th_m(x)),
	\]
	with universal crossed product $(C_0(\O)\rtimes_\a \Z^k, j_{C_0(\O)},j_{\Z^k})$.  Yeend
asserts that the topological $k$-graph $C^*$-algebra is isomorphic to this crossed product.  The
main result in this section, Theorem~\ref{isomthm}, generalizes this to the setting where the maps
are local homeomorphisms that are not necessarily surjective.
\end{rem}

\begin{prop}\label{coisomIScp}
Let $(\O,\{T_i\}_{i=1}^k)$ be a topological dynamical system {that satisfies condition (UBC)} and let $(C_0(\O),\N^k,\a,L)$ be the
Exel-Larsen system described in Subsection~\ref{subsec:exlar}.  Let $X^{Lar}$ be
the product system over $\N^k$ of $C_0(\O)$-cor\-re\-spond\-en\-ces from Definition~\ref{def:XLar}.
Let $K=\{K_m\}_{m\in\N^k}$ be the family of ideals defined by \eqref{eqn:coisomsets}.  Let $\psi:X^{Lar}
\to B$ be a (Toeplitz) representation of $X^{Lar}$ in a $C^*$-algebra $B$.  Then $\psi$ is Cuntz-Pimsner
covariant in the sense of \eqref{eqn:CP-K} if and only if it is coisometric on $K$.
\end{prop}

\begin{proof}
In the proof of Lemma~\ref{lem:FvsK}, $\a_m(C_0(\O))C_0(\O)=C_0(\O)$ hence
\begin{align*}
	K_m &= \overline{C_0(\O)\a_m(C_0(\O))C_0(\O)}\cap\phi_m^{-1}(\KK(X^{Lar}_m)) \\
	&= C_0(\O)\cap\phi_m^{-1}(\KK(X^{Lar}_m)) \\
	&= \phi_m^{-1}(\KK(X^{Lar}_m)).
\end{align*}
Recall from the construction of $X^{Lar}$ in Definition~\ref{def:XLar} that the left action on each $X^{Lar}_m$ is
given by multiplication. Thus $\phi_m$ is injective so that $(\ker\phi_m)^\perp=C_0(\O)$ and hence
coisometric on $K$ is equivalent to \eqref{eqn:CP-K}.\end{proof}

\begin{cor}\label{cor:lar}
Let $(\O,\{T_i\}_{i=1}^k)$ be a topological dynamical system {satisfying condition (UBC)} and let $(C_0(\O),\N^k,\a,L)$ be the
Exel-Larsen system described in Subsection~\ref{subsec:exlar}.  Let $X^{Lar}$ be the product
system over $\N^k$ of $C_0(\O)$-\csps~from Definition~\ref{def:XLar}.  Then
\[
C_0(\O)\rtimes_{\a,L}\N^k \cong \OO_{X^{Lar}}.
\]
\end{cor}

\begin{proof}
Since a representation $\psi:X\to B$ is coisometric on $K$, where $K=\{K_m\}_{m\in \N^k}$ is the
family of ideals defined by \eqref{eqn:coisomsets}, if and only if it is Cuntz-Pimsner covariant in the
sense of \eqref{eqn:CP-K}, it follows that $j^{Fow}:X^{Lar}\to\OO_{X^{Lar}}$ is coisometric on $K$
and $j^{Lar}$ satisfies \eqref{eqn:CP-K}.  It follows from the universal properties of $\OO_{X^{Lar}}$
and the Larsen crossed product $C_0(\O)\rtimes_{\a,L}\N^k$ that there are unique surjective
homomorphisms
\begin{align*}
	\Pi_{j^{Lar}}&:\OO_{X^{Lar}}\to C_0(\O)\rtimes_{\a,L}\N^k \\
	\Pi_{j^{Fow}}&:C_0(\O)\rtimes_{\a,L}\N^k \to \OO_{X^{Lar}}
\end{align*}
such that $j^{Lar}=\Pi_{j^{Lar}}\circ j^{Fow}$ and $j^{Fow}=\Pi_{j^{Fow}}\circ j^{Lar}$.  Since $\OO_
{X^{Lar}}$ and $C_0(\O)\rtimes_{\a,L}\N^k$ are generated by $j^{Fow}$ and $J^{Lar}$
respectively, it follows that $\Pi_{j^{Lar}}$ and $\Pi_{j^{Fow}}$ take generators to generators and
hence
\[
C_0(\O)\rtimes_{\a,L}\N^k\cong\OO_{X^{Lar}}.
\]
\end{proof}

We now show that for any representation $\psi:X^\L\to B$, Cuntz-Pimsner covariance in the sense
of \eqref{eqn:CP-K} is equivalent to CNP-covariance.  In order for CNP-covariance to make sense
for a representation of $X^\L$, we must have that $X^\L$ is compactly aligned.  Recall that $X^\L$
is the topological $k$-graph cor\-re\-spond\-ence associated to the topological $k$-graph $\L=(\L
(\O,\Th),d)$ which we showed in Proposition \ref{prop:lambdaca} is compactly aligned.  By \cite
[Proposition 5.15]{clsv}, since $\L$ is compactly aligned, so is $X^\L$.

\begin{prop}\label{prop:CPisCNP}
Let $(\O,\{T_i\}_{i=1}^k)$ be a topological dynamical system and let $\L$ be the associated
topological $k$-graph.  Let $X^\L$ be the topological $k$-graph \csp.  Let $\psi:X^\Lambda\to B$ be a (Toeplitz) representation of $X^\Lambda$ in a $C^*$-algebra $B$.  Then $\psi$ is Cuntz-Pimsner covariant in the sense of \eqref{eqn:CP-K} if and only if it is CNP-covariant.
\end{prop}

We would like to apply \cite[Corollary 5.2]{sy} to obtain the desired result.  In order to do so, we
need to establish that the left action on each fibre is by compact operators.

\begin{lem} \label{ComOp}
Let $(\O, \{T_i\}^k_{i=1})$ be a topological dynamical system and $\L$ be the associated
topological $k$-graph.  Then the left action of $C_0(\O)$ on each fibre $X^\L_m$ of the topological
$k$-graph \csp~is by compact operators.
\end{lem}

\begin{proof}
It follows from \cite[Proposition 1.24]{katsgen}, that $\phi_m^{-1}(\KK(X_m))=C_0(\O_{fin})$
where
\[
\O_{fin}=\{v\in\O: v\textrm{ has a neighborhood $V$ such that } r_m^{-1}(V) \textrm{ is compact}\}.
\]
Since $\O$ is a locally compact space, every point $v\in\O$ has a compact neighborhood $V$.  The
range map is given by projection onto $\O$ so that $r_m\inv(V)=\{m\}\times V$, which is compact.  It
follows that $\O_{fin}=\O$.
\end{proof}

\begin{proof}[Proof of \propref{prop:CPisCNP}]
Recall that $(\Z^k,\N^k)$ is a quasi-lattice ordered group such that each pair $s,t\in\N^k$ has a
least upper bound and that $X^\L$ is compactly aligned.  It follows from the construction of $X^\L$
in Definition~\ref{def:topkgraphcorresp} that the left action on each fibre is given by multiplication
and is therefore injective.  By \lemref{ComOp}, the left action on each fibre is by compact operators.
Then by \cite[Corollary 5.2]{sy}, $\psi$ is CNP-covariant if and only if
\[
\psi^{(m)}\circ\phi_m =\psi_0 \all m\in\N^k.
\]
It follows that $\psi$ is CNP-covariant if and only if $\psi$ satisfies \eqref{eqn:CP-K}.
\end{proof}

\begin{cor}\label{cor:lambda}
Let $(\O,\{T_i\}_{i=1}^k)$ be a topological dynamical system and let $\L$ be the topological $k$-graph described in Subsection~\ref{subsec:lambda}.  Let $X^\L$ be the associated topological $k$-graph \csp~as in Definition~\ref{def:topkgraphcorresp}.  Then
\[
\NO_{X^\L} \cong \OO_{X^\L}.
\]
\end{cor}

\begin{proof}
Since a representation $\psi:X^\Lambda\to B$ is Cuntz-Pimsner covariant in the sense of \eqref{eqn:CP-K} if and only if it is CNP-covariant, it follows that $j^{Fow}:X^\L \to\OO_{X^\L}$ is CNP-covariant and $j^{CNP}$ satisfies \eqref{eqn:CP-K}.  It follows from the universal properties of $\OO_{X^\L}$ and $\NO_{X^\L}$ that there are unique surjective homomorphisms
\begin{align*}
	\Pi_{j^{CNP}}&:\OO_{X^\L}\to \NO_{X^\L} \\
	\Pi_{j^{Fow}}&:\NO_{X^\L} \to \OO_{X^\L}
\end{align*}
such that $j^{CNP}=\Pi_{j^{CNP}}\circ j^{Fow}$ and $j^{Fow}=\Pi_{j^{Fow}}\circ j^{CNP}$.  Since $
\OO_{X^\L}$ and $\NO_{X^\L}$ are generated by $j^{Fow}$ and $j^{CNP}$ respectively, it follows
that $\Pi_{j^{CNP}}$ and $\Pi_{j^{Fow}}$ take generators to generators and hence
\[
\NO_{X^\L}\cong\OO_{X^\L}.
\]
\end{proof}

\begin{thm}\label{isomthm}
Let $(\O, \{T_i\}^k_{i=1})$ be a topological dynamical system {satisfying condition (UBC)}.  Let $C_0(\O)\rtimes_{\a,L}\N^k$ be the Larsen crossed product and $C^*(\L)$ be the topological $k$-graph $C^*$-algebra associated to $(\O, \{T_i\}^k_{i=1})$. Then
\[
C^*(\L)\cong C_0(\O)\rtimes_{\a,L}\N^k.
\]
\end{thm}

\begin{proof}
By \cite[Theorem 5.20]{clsv}, we have that $C^*(\L)\cong\NO_{X^\L}$.  By \thmref
{prop:prodsysisom}, $X^{Lar}\cong X^\Lambda$ and hence the Cuntz-Pimsner algebras $\mathcal{O}_{X^{Lar}}$ and $\mathcal{O}_{X^\Lambda}$ are isomorphic.  Thus,
\[
C^*(\L)\cong\NO_{X^\L}\cong\OO_{X^\L}\cong\OO_{X^{Lar}}\cong C_0(\O)\rtimes_{\a,L}\N^k.
\]
\end{proof}

\begin{ex}
It is known that $C^*(\O_k)\cong\KK(\ell^2(\N^k))$.  The result above, together with our description
of the Exel-Larsen system $(C_0(\N^k),\N^k,\a,L)$ as in Example~\ref{ex:kgraphshift:EL} gives that $\KK
(\ell^2(\N^k))\cong C_0(\N^k)\rtimes_{\a,L}\N^k$.
\end{ex}

\providecommand{\bysame}{\leavevmode\hbox to3em{\hrulefill}\thinspace}
\providecommand{\MR}{\relax\ifhmode\unskip\space\fi MR }
\providecommand{\MRhref}[2]{%
  \href{http://www.ams.org/mathscinet-getitem?mr=#1}{#2}
}
\providecommand{\href}[2]{#2}

\end{document}